\documentclass[12pt]{amsart}
\usepackage{amsfonts, pdfsync}
\usepackage{graphicx, comment}
\usepackage{times}
\usepackage{amsmath}
\usepackage{amssymb}
\usepackage{latexsym}
\usepackage{amsthm}
\newtheorem{theorem}{Theorem}[section]
\newtheorem{lemma}[theorem]{Lemma}

\newtheorem{definition}[theorem]{Definition}

\newtheorem{remark}[theorem]{Remark}
\newtheorem{open}[theorem]{Open Problem}
\newcommand{\As}{\mathcal{A}_{\mathrm{s}}}
\newcommand{\Aw}{\mathcal{A}_{\mathrm{w}}}
\newcommand{\A}{\mathcal{A}}

\newcommand{\ds}{\mathrm{d}_{\mathrm{s}}}
\newcommand{\dw}{\mathrm{d}_{\mathrm{w}}}

\newcommand{\Dc}{\mathcal{E}}

\newcommand{\w}{\mathrm{w}}
\newcommand{\Xw}{X_{\mathrm{w}}}
\newcommand{\Xs}{X_{\mathrm{s}}}
\newcommand{\Hw}{H_{\mathrm{w}}}
\newcommand{\Ab}{\mathcal{A}_{\bullet}}
\newcommand{\db}{\mathrm{d}_{\bullet}}
\newcommand{\dd}{\mathrm{d}}

\newcommand{\wb}{{\omega}_{\bullet}}
\newcommand{\ww}{{\omega}_{\mathrm{w}}}
\newcommand{\ws}{{\omega}_{\mathrm{s}}}

\newcommand{\ddt}{\frac{d}{dt}}

\newcommand{\Rc}{\widetilde{R}}

\def\setmarsing{
\oddsidemargin-0in
\evensidemargin-0in
\textwidth5.5in
\textheight8in
}
\setmarsing

\begin{document}
\title{Uniform global attractors for the nonautonomous 3D Navier-Stokes equations}
\author{Alexey Cheskidov}
\thanks{A.C. was partially supported by the NSF grant DMS--1108864}
\address[A. Cheskidov]
{Department of Mathematics, Statistics and Computer Science\\ University of Illinois at Chicago
\\851 S Morgan St, Chicago, IL 60607-7045, USA}
\email{acheskid@math.uic.edu}

\author{Songsong Lu}
\thanks{S.L. was partially supported
by Specialized Research Fund for the Doctoral Program of Higher Education (200805581025), NSFC 11001279 and the Fundamental Research Funds for the Central Universities (11lgpy27).}
\address[S. Lu]
{Department of Mathematics, Sun Yat-sen University
\\Guangzhou, 510275, P.R. China}
\email{songsong\_lu@yahoo.com}

\begin{abstract}
We obtain the existence and the structure of the weak uniform (with respect to the initial time) global attractor and construct a trajectory attractor for the 3D
Navier-Stokes equations (NSE) with a fixed time-dependent force satisfying a
translation boundedness condition. Moreover, we show that if
the force is normal and every complete bounded solution is strongly continuous, then the uniform global attractor is strong, strongly compact, and solutions converge strongly toward the trajectory attractor.
Our method is based on taking a closure of
the autonomous evolutionary system without uniqueness, whose trajectories are solutions to the nonautonomous 3D NSE.
The established framework is general and can also be applied to other
nonautonomous dissipative partial differential equations for which the uniqueness of
solutions might not hold. It is not known whether previous frameworks can
also be applied in such cases as we indicate in open problems related
to the question of uniqueness of the  Leray-Hopf weak solutions.


\end{abstract}

\maketitle

\section{Introduction}

The theory of uniform attractors of nonautonomous infinite-dimensional dissipative dynamical system bears its roots in the work of Haraux \cite{Ha91}, who defined the
uniform global attractor as a minimal closed set which attracts
all the trajectories starting from a bounded set uniformly with respect to (w.r.t.)
the initial time. This naturally generalizes the notion of a global attractor to nonautonomous dynamical systems. In this paper we will present a method for obtaining the structure of the uniform  global attractor of a general nonautonomous
system. In particular, we obtain the existence and the structure of the weak uniform global attractor and construct a trajectory attractor for the 3D Navier-Stokes equations (NSE) with a fixed time-dependent force satisfying a
translation boundedness condition.

Previous studies of uniform global attractors mostly used tools developed by
Chepyzhov and Vishik \cite{CV94, CV02}.
Their framework was based on the use of the so-called time symbol (e.g. the external force in 3D NSE) and constructing a symbol space as a suitable
closure of the translation family of the original symbol.
To describe the structure of the uniform global attractor they introduced an auxiliary notion of a uniform w.r.t. the symbol space attractor.
However, this method requires a strong condition on the force and
only provides the structure of the uniform w.r.t. the symbol space attractor,
which does not always have to coincide with the original uniform global
attractor (see Open Problem \ref{op:attractor}).
In this paper we present a different approach that deals directly with the notion of a uniform global attractor and is  based on taking a closure of the family of trajectories
of the system, which does not change the uniform global attractor.
The established method
is general and can be applied to any nonautonomous dissipative PDE.

Since the pioneering work of
Leray, the problem of regularity of the 3D NSE has been a subject of
serious investigation and still poses an important challenge for
mathematicians. Due to the lack of a uniqueness proof, it is not
known whether the 3D NSE possesses, for the autonomous case,  a semigroup of, or for the nonautonomous case, a process of  solution operators. Therefore, a classical theory of semigroup or process
\cite{H88,T88, L91, SY02, CV02} cannot be used for this system. A mathematical object describing long time behavior of an autonomous
3D NSE is a (weak) global attractor, a notion that goes all the way back to
the seminal work by Foias and Temam \cite{FT87}.

The goal of this paper is to study the long time behavior of the
3D NSE with a fixed time-dependent force in the physical space without making any assumptions
on weak solutions. Moreover, we assume that the force only satisfies a translation boundedness condition, which is the weakest condition that
guaranties the existence of a bounded uniform absorbing ball. In order to obtain
the  structure of the weak uniform global
attractor we will consider an autonomous evolutionary system without uniqueness, whose trajectories are solutions to the nonautonomous 3D NSE.
 The evolutionary system $\mathcal{E}$
was first introduced  in \cite{CF06} to study a
weak global attractor and a trajectory attractor for the autonomous 3D NSE, and then the
theory was developed further in \cite{C09} to make it
applicable to arbitrary autonomous dissipative PDE without uniqueness.
In particular, it was shown that the global attractor consists of points on complete bounded trajectories under an assumption (see \={A1}) satisfied by autonomous PDEs.
The evolutionary system is close to
Ball's generalized semiflow  \cite{B98}, but due to relaxed assumptions on the trajectories,
the Leray-Hopf weak solutions of the 3D NSE always form an evolutionary system
regardless whether they lose regularity or not. The advantage of this framework
lies in a simultaneous use of weak and strong metric, which makes it applicable to
any other PDE for which the uniqueness of solutions may
be in limbo.  Another canonical abstract framework for studying dynamical systems without
uniqueness in the phase space \cite{MV98} also requires various assumptions on the
trajectories (see \cite{CMR03} for comparison to \cite{B98} ).

In \cite{CL09} the authors already generalized  the framework of the evolutionary system to study the long time
behavior of nonautonomous dynamical systems without uniqueness. In this paper we develop the theory further and introduce a ``closure of the evolutionary system''
in order to obtain the structure of the  uniform
global attractor. This method  avoids the necessity of constructing a symbol space and works for systems without uniqueness.

The paper is organized as follows. In Section \ref{s:ES} we briefly
recall the theory of evolutionary system originally designed for
autonomous systems, define  a nonautonomous evolutionary system, and reduce it to an autonomous system. Then we consider classical cases of a
process and a family of processes and show how they define evolutionary systems.
In particular, when the evolutionary system is defined by a process, the uniform global attractor for the evolutionary system is identical to the uniform  global attractor for the process. Hence, using the theory
developed in this paper, we can describe the structure of the uniform global attractor of a general process, which was mentioned as an open problem in \cite{CV94, CV02}.

Section \ref{s:closing} is mainly concerned with the existence and the
structure of the weak and strong uniform global attractors for the nonautonomous evolutionary system. To this end, we consider a closure of the evolutionary system and prove
that its weak uniform global attractor is identical to the one for the original
evolutionary system. We then apply the theory
developed in \cite{C09} to the closure of the evolutionary system to obtain
various properties of the uniform global attractor for the original system.

In Subsection \ref{s:subES} we use our framework to examine the notion of uniform w.r.t. symbol space global attractor.
For this we assume that a suitable symbol space $\bar{\Sigma}$  is provided, and
consider a nonautonomous evolutionary system with such symbol space satisfying the uniqueness condition, i.e., we assume that for a fixed symbol there exists only one trajectory starting at a given point. We then study the relation between the system and its evolutionary subsystem whose symbol space $\Sigma$ is a dense subset of the former symbol space $\bar{\Sigma}$. Assuming that the system satisfies
an  additional condition
(which translates into the strong translation compactness condition on the force in the case of the 3D NSE), we show that it has the same weak uniform global attractor as its subsystem. Therefore, we show that the uniform w.r.t. $\bar{\Sigma}$ attractor coincides with the uniform global attractor.
Furthermore, if the evolutionary system is asymptotically compact, then the weak uniform global attractor is in fact the strong uniform global attractor.  The results in this subsection generalize those in \cite{CV02, LWZ05, Lu06, Lu07}. It is worth to mention that  our framework  is also useful for
more  general systems with uniqueness, where a symbol space may not exist, since we avoid constructing
such a space.

In Section \ref{s:traattra} we study the notion of a trajectory attractor that was first introduced in \cite{Se96} and further studied in \cite{CV97, CV02, SY02}.
We use the tools in the preceding sections to show the existence
and the structure of the trajectory attractor of a nonautonomous
evolutionary system such as the 3D NSE. Note that such a result could not be
obtained with previous frameworks (see Open Problem \ref{op:traattractor}).
We also show a relation of the trajectory attractor to a uniform global attractor. Moreover, in the case where all complete bounded trajectories are strongly continuous, a strong convergence of the trajectories to the trajectory attractor
is proved.

In Section \ref{3DNSE} we consider the 3D NSE with a translation bounded in
$L^2_\mathrm{loc}(\mathbb{R};V')$ fixed time-dependent external  force $g_0$.
We show that the Leray-Hopf weak solutions form an evolutionary system investigated in  Sections \ref{s:closing}, and
therefore all the results obtained in that section hold for the 3D NSE.
In particular, we obtain the structure of the weak uniform global attractor $\A$.
Notice again that we only require the weakest boundedness condition on the force and
make no assumptions on the solutions of the 3D NSE.
In addition, we show that if
the force $g_0$ is normal and every complete bounded solution is strongly continuous, then the weak uniform global attractor is strong, strongly compact, and solutions converge strongly toward the trajectory attractor.
The normality condition on the external force, introduced in \cite{LWZ05} and
Lu \cite{Lu06},  is weaker than the usual strong translation compactness condition (see \cite{LWZ05}), which is generally required in applications of Chepyzhov and Vishik's approach \cite{CV94,CV02}.

Evolutionary systems are constructed without a suitable symbol space  in Section \ref{3DNSE}. In Section \ref{s:OP} we construct a uniform w.r.t. symbol space global attractor for the
3D NSE. For this we first have to impose a stronger condition on the external force,
namely, we assume that the force is strongly translation compact. Then we can
find a suitable closure of the symbol space $\bar{\Sigma}$ for which the corresponding
evolutionary system enjoys the desired compactness property, and hence
obtain the structure of the uniform w.r.t. symbol space $\bar{\Sigma}$ attractor
$\Aw^{\bar{\Sigma}}$. However, this attractor might not coincide with the
uniform global attractor $\A$ if Leray-Hopf weak solutions are not unique (see
the Open Problem \ref{op:attractor}), which illustrates a limitation of the framework of uniform w.r.t. symbol global attractor put forward in \cite{CV94} (see also \cite{CV02}).
It is still not clear how to obtain the structure of the uniform global attractor $\A$
using the notion of  uniform w.r.t. symbol global attractor $\Aw^{\bar{\Sigma}}$ for the 3D NSE or other systems without
uniqueness, or where the uniqueness is not known. Similarly, the trajectory attractors constructed in \cite{Se96}
and \cite{CV97,CV02} for the 3D NSE or other systems without uniqueness,
might be bigger than the one we constructed in this section (see Open Problem \ref{op:traattractor}).

\section{Evolutionary system}\label{s:ES}
\subsection{Autonomous case}
Here we recall the basic definitions and results on evolutionary
systems (see \cite{C09} for details). Let $(X,\ds(\cdot,\cdot))$ be
a metric space endowed with a metric $\ds$, which will be referred
to as a strong metric. Let $\dw(\cdot, \cdot)$ be another metric on
$X$ satisfying the following conditions:
\begin{enumerate}
\item $X$ is $\dw$-compact.
\item If $\ds(u_n, v_n) \to 0$ as $n \to \infty$ for some
$u_n, v_n \in X$, then $\dw(u_n, v_n) \to 0$ as $n \to \infty$.
\end{enumerate}
Due to the property 2, $\dw(\cdot,\cdot)$ will be referred to as a weak metric on $X$. Denote by $\overline{A}^{\bullet}$ the closure of a set $A\subset X$
in the topology generated by $\db$.
Note that any strongly compact ($\ds$-compact) set is weakly compact
($\dw$-compact), and any weakly closed set is strongly closed.

Let $C([a, b];X_\bullet)$, where $\bullet = \mathrm{s}$ or $\mathrm{w}$, be the space of $\db$-continuous $X$-valued
functions on $[a, b]$ endowed with the metric
\[
\dd_{C([a, b];X_\bullet)}(u,v) := \sup_{t\in[a,b]}\db(u(t),v(t)).
\]
Let also $C([a, \infty);X_\bullet)$ be the space of $\db$-continuous
$X$-valued functions on $[a, \infty)$
endowed with the metric
\[
\dd_{C([a, \infty);X_\bullet)}(u,v) := \sum_{T\in \mathbb{N}} \frac{1}{2^T} \frac{\sup\{\db(u(t),v(t)):a\leq t\leq a+T\}}
{1+\sup\{\db(u(t),v(t)):a\leq t\leq a+T\}}.
\]
Note that the convergence in $C([a, \infty);X_\bullet)$ is equivalent to uniform convergence on compact sets.
Let
\[
\mathcal{T} := \{ I: \ I=[T,\infty) \subset \mathbb{R}, \mbox{ or }
I=(-\infty, \infty) \},
\]
and for each $I \subset \mathcal{T}$, let $\mathcal{F}(I)$ denote
the set of all $X$-valued functions on $I$.
Now we define an evolutionary system $\Dc$ as follows.
\begin{definition} \label{Dc}
A map $\Dc$ that associates to each $I\in \mathcal{T}$ a subset
$\Dc(I) \subset \mathcal{F}(I)$ will be called an evolutionary system if
the following conditions are satisfied:
\begin{enumerate}
\item $\Dc([0,\infty)) \ne \emptyset$.
\item
$\Dc(I+s)=\{u(\cdot): \ u(\cdot +s) \in \Dc(I) \}$ for
all $s \in \mathbb{R}$.
\item $\{u(\cdot)|_{I_2} : u(\cdot) \in \Dc(I_1)\}
\subset \Dc(I_2)$ for all
pairs $I_1,I_2 \in \mathcal{T}$, such that $I_2 \subset I_1$.
\item
$\Dc((-\infty , \infty)) = \{u(\cdot) : \ u(\cdot)|_{[T,\infty)}
\in \Dc([T, \infty)) \ \forall T \in \mathbb{R} \}.$
\end{enumerate}
\end{definition}

We will refer to $\Dc(I)$ as the set of all trajectories
on the time interval $I$. Trajectories in $\Dc((-\infty,\infty))$ are called complete.
Let $P(X)$ be the set of all subsets of $X$.
For every $t \geq 0$, define a map
\begin{eqnarray*}
&R(t):P(X) \to P(X),&\\
&R(t)A := \{u(t): u(0)\in A, u \in \Dc([0,\infty))\}, \qquad
A \subset X.&
\end{eqnarray*}
Note that the assumptions on $\Dc$ imply that $R(t)$ enjoys
the following property:
\begin{equation} \label{eq:propR(T)}
R(t+s)A \subset R(t)R(s)A, \qquad A \subset X,\quad t,s \geq 0.
\end{equation}
\begin{definition} \label{DAttractor}
A set $\Ab\subset X$ is a $\db$-global attractor
($\bullet=\mathrm{s}, \mathrm{w}$) if $\Ab$ is a minimal set
that is
\begin{enumerate} \item $\db$-closed. \item $\db$-attracting: for
any $B\subset X$ and $\epsilon>0$, there exists $t_0$, such that
\[R(t)B\subset B_\bullet(\Ab,\epsilon):=\{u:
\inf_{x\in \Ab}\db(u,x)<\epsilon\},\quad \forall\, t\geq t_0.
\]
\end{enumerate}
\end{definition}
\begin{definition}The $\wb$-limit ($\bullet=\mathrm s, \mathrm w$) of a set $A\subset X$ is
\[\wb(A):=\bigcap_{T\ge0}\overline{\bigcup_{t\ge T}R(t)A}^{\bullet}.\]\end{definition}

An equivalent definition of the $\wb$-limit set is given by
\[
\begin{split}
\wb(A)=\{&x\in X: \mbox{ there exist sequences }
t_n\rightarrow\infty \mbox{ as } n\rightarrow \infty\mbox{ and } x_n \in R(t_n)A,\\
& \mbox{such that } x_n\rightarrow x \mbox{ in } \db\mbox{-metric as
} n\rightarrow \infty\}.
\end{split}
\]

In order to extend the notion of invariance from a semiflow to an
evolutionary system we use the following mapping:
\[
\Rc(t) A := \{u(t): u(0) \in A, u \in \Dc((-\infty,\infty))\}, \qquad A \subset X,
\ t \in \mathbb{R}.
\]
\begin{definition} A set $A\subset X$ is positively invariant if
\[
\Rc(t) A\subset A, \qquad \forall t\geq 0.
\]
$A$ is invariant if
\[
\Rc(t) A= A, \qquad \forall t\geq 0.
\]
$A$ is quasi-invariant if for every $a\in A$ there exists a complete trajectory
$u \in \Dc((-\infty,\infty))$ with $u(0)=a$ and $u(t) \in A$ for all $t\in \mathbb{R}$.
\end{definition}
Note that in the case of semiflows the notions of invariance and quasi-invariance coincide
with the classical definition of invariance. This is not the case for evolutionary systems
without uniqueness. For instance, the weak global attractor for the 3D Navier-Stokes equations
consists of points on complete bounded trajectories. However, due to the lack of concatenation,
it is not known whether all the trajectories starting from the global attractor have to stay on the attractor.
\begin{definition}The evolutionary system $\Dc$ is asymptotically
compact if for any $t_k\rightarrow\infty $ and any $x_k\in R(t_k)X$,
the sequence $\{x_k\}$ is relatively strongly compact.
\end{definition}

Below are some additional assumptions that we will impose on $\Dc$ in
some cases.
\begin{itemize}
\item[\={A1}] $\Dc([0,\infty))$ is a compact set in $C([0,\infty); \Xw)$.
\item[\={A2}] (Energy inequality) Assume that $X$ is a set in some
Banach space $H$ satisfying the Radon-Riesz property (see below)
with the norm denoted by $|\cdot|$, such
that $\ds(x,y)=|x-y|$ for $x,y \in X$ and $\dw$ induces the weak
topology on $X$. Assume also that for any
$\epsilon >0$, there exists $\delta$, such that for every $u \in
\Dc([0,\infty))$ and $t>0$,
\[
|u(t)| \leq |u(t_0)| + \epsilon,
\]
for $t_0$ a.e. in $ (t-\delta, t)$.
\item[\={A3}] (Strong convergence a.e.) Let $u,u_n \in \Dc([0,\infty))$, be such that
$u_n \to u$ in $C([0, T];\Xw)$ for some $T>0$. Then $u_n(t) \to
u(t)$ strongly a.e. in $[0,T]$.
\end{itemize}

\begin{remark}
A Banach space $H$ is said to satisfy the Radon-Riesz property when a sequence
converges if and only if it converges weakly and the norms of the elements of the sequence converge to the
norm of the weak limit.
In many applications $X$ is a bounded closed set in a uniformly convex separable
Banach space $H$. Then the weak topology of $H$ is metrizable on $X$, and $X$ is
compact with respect to such a metric $\dw$. Moreover, the Radon-Riesz property is automatically satisfied in this case.
\end{remark}

\begin{theorem} \cite{C09}\label{t:weakA}
Let $\Dc$ be an evolutionary system. Then
\begin{itemize}\item [1.]
The weak global attractor $\Aw$ exists.\end{itemize} Furthermore, if
$\Dc$ satisfies \={A1}, then
\begin{itemize}
\item [2.]$ \Aw =\ww(X)=\ws(X)=\{ u_0: \ u_0=u(0) \mbox{ for some } u \in
\Dc((-\infty, \infty))\} $.
\item [3.]$\Aw$ is the maximal invariant
and maximal quasi-invariant set.
\item [4.](Weak uniform tracking
property) For any $\epsilon >0$, there exists $t_0$, such that for
any $t^*>t_0$, every trajectory $u \in \Dc([0,\infty))$ satisfies $
\dd_{C([t^*,\infty);\Xw)} (u, v) < \epsilon, $ for some complete
trajectory $v \in \Dc((-\infty,\infty))$.
\end{itemize}
\end{theorem}

\begin{theorem} \cite{C09}\label{t:strongA}
Let $\Dc$ be an asymptotically compact evolutionary system.
Then
\begin{enumerate}
\item[1.]
The strong global attractor $\As$ exists, it is strongly compact, and
$\As=\Aw$.
\end{enumerate}
Furthermore, if $\Dc$ satisfies \={A1}, then
\begin{enumerate}
\item[2.] (Strong uniform tracking property)
for any $\epsilon >0$ and $T>0$, there exists $t_0$, such that for
any $t^*>t_0$, every trajectory $u \in \Dc([0,\infty))$ satisfies
$\ds(u(t), v(t)) < \epsilon$, $\forall t\in [t^*,t^*+T]$, for some
complete trajectory $v \in \Dc((-\infty,\infty))$.
\end{enumerate}
\end{theorem}
\begin{theorem} \cite{C09}\label{t:AComp}
Let $\Dc$ be an evolutionary system satisfying \={A1}, \={A2}, and \={A3} and
such that every complete trajectory is strongly continuous.
Then $\Dc$ is asymptotically compact.
\end{theorem}

\subsection{Nonautonomous case}\label{S:nonauto}
In this subsection we will show that the notion  of evolutionary
system is  naturally applicable to a nonautonomous system.

 Let $\Sigma $ be a parameter set and $\left\{T(s)| s\ge
0\right\}$ be a family of operators acting on $\Sigma$ satisfying
 $T(s)\Sigma=\Sigma$, $\forall s\ge0$. Any element $\sigma\in \Sigma $ will be  called (time) symbol and $ \Sigma $ will be called (time) symbol
 space. For instance, in many applications $\{T(s)\}$ is the translation
 semigroup and $ \Sigma$ is  the translation family
 of the time-dependent items of the considered system or its closure in some appropriate topological space (for more examples see \cite{CV02}, the appendix in \cite{CLR12}).
\begin{definition} \label{d:Dc0}A family of maps $\mathcal E_\sigma $, $\sigma \in \Sigma $ that for every $\sigma\in \Sigma$
associates to each $I\in \mathcal T$ a subset $\mathcal
E_\sigma(I)\subset \mathcal F(I)$ will be called a nonautonomous
evolutionary system  if the following conditions are satisfied:
\begin{enumerate}
\item $\mathcal E_\sigma ([\tau, \infty ))\neq \emptyset,
\forall\, \tau \in \mathbb R$. \item $\mathcal E_ \sigma
(I+s)=\{u(\cdot): u(\cdot+s)\in  \mathcal E_{T(s)\sigma}(I)\},
\forall s\ge 0$. \item $\{u(\cdot)|_{I_2}:u(\cdot)\in \mathcal
E_\sigma(I_1)\}\subset \mathcal E_\sigma(I_2)$, $\forall\, I_1$,
$I_2\in \mathcal T$, $I_2\subset I_1$. \item $\mathcal
E_\sigma((-\infty,\infty))=\{u(\cdot):u(\cdot)|_{[\tau,\infty)}\in
\mathcal E_\sigma([\tau,\infty)),\forall\,\tau \in \mathbb R\}$.
\end{enumerate}
\end{definition}

We will refer to $\mathcal E_\sigma(I)$ as the set of all
trajectories with respect to (w.r.t.) the symbol $\sigma$ on the
time interval $I$. Trajectories in $\mathcal E_\sigma((-\infty,
\infty))$ will be called complete w.r.t. $\sigma$.  For every
$t\ge \tau$, $\tau\in \mathbb R$, $\sigma\in \Sigma$, define a map
\begin{align*}R_\sigma(t, \tau ): P(X)&\rightarrow P(X),\\
R_\sigma(t,\tau )A:=\{u(t):u(\tau )\in A, u&\in \mathcal E_\sigma
([\tau,\infty)\}, \ \ A\subset X.
\end{align*}
Similarly,  the assumptions on $\mathcal E_\sigma$, $\sigma\in
\Sigma $  imply that $R_\sigma (t,\tau)$ enjoys  the following
property:
\begin{equation}\label{eq:propRsigma(T)}R_\sigma
(t,\tau)A\subset R_\sigma(t, s)R_\sigma(s,\tau)A, \ \ A\subset
X, \qquad \forall t\ge s \ge \tau , \tau\in \mathbb R.\end{equation}

Let us now show how a nonautonomous evolutionary system can
be defined in the classical case where the uniqueness of trajectories
holds.

Let $H$ be a phase space (a separable reflexive Banach space).
Consider a process of a two-parameter family of single-valued operators
$U_{\sigma_0}(t,\tau): H\rightarrow H$, satisfying the following conditions:
\begin{equation}\label{definitionprocess}\begin{split}
U_{\sigma_0}(t,s)\circ U_{\sigma_0}(s,\tau)&=U_{\sigma_0}(t,\tau),
\qquad \forall\, t\ge s\ge\tau,\ \tau\in\mathbb{R},
\\U_{\sigma_0}(\tau,\tau)&=\text{Identity
operator,}\qquad  \tau\in\mathbb{R}.
\end{split}
\end{equation}
Here $\sigma_0$ is a fixed symbol, which is usually the collection
of all time-dependent terms of a considered system. So we assume
that it is a function on $\mathbb R$ with values in some space. A
trajectory $u$ on $[\tau, \infty)$ is a mapping from $[\tau,\infty)$
to $H$, such that
\begin{equation}\label{solutiontoprocess} u(t)=U_{\sigma_0}
(t,\tau)u(\tau),\qquad t\ge\tau.
\end{equation}
 A ball $B\subset H$ is called a uniformly (w.r.t. $\tau\in \mathbb R$)
absorbing ball if for any  bounded set $A\subset H$, there exists a
$t_0=t_0(A)$, such that,
\begin{equation}\label{d:initial absorbing}\bigcup_{\tau\in \mathbb R}\bigcup_{t\geq
t_0}U_{\sigma_0}(t+\tau,\tau)A\subset B.
\end{equation} Assume that
the  process is dissipative, i.e., there exists a uniformly (w.r.t.
$\tau\in \mathbb R$) absorbing ball $B$. Since we are interested in
a long-time behavior of solutions, it is enough to consider a
restriction of the  process to $B$. A uniform (w.r.t. $\tau\in
\mathbb R$) attractor $\mathcal A$ of  the process is a minimal
closed set satisfying that, for any $A\subset B$ and $\epsilon>0$
there exists  $t_0=t_0(\epsilon,A)$, such that
\begin{equation}\label{d:initial attractor}
\bigcup_{\tau\in \mathbb R}\bigcup_{t\geq
t_0}U_{\sigma_0}(t+\tau,\tau)A\subset B_{H}(\mathcal A, \epsilon).
\end{equation}
Now denote by $\Sigma$ the translation family
$\{\sigma_0(\cdot+h)|h\in \mathbb R\}$ of $\sigma_0$ and define
\begin{equation}\label{d:process}U_{\sigma_0(\cdot+h)}(t,\tau):=U_{\sigma_0(\cdot)}(t+h,\tau+h),\qquad \forall\,
t\ge\tau,\ \tau\in\mathbb{R},\ h\in\mathbb R.
\end{equation}Due to the
uniqueness of the trajectories, for any $\sigma\in \Sigma$, the
family of operators $U_{\sigma}(t,\tau)$ define a process, i.e.,
(\ref{definitionprocess}) is valid with $\sigma$ substituted for
$\sigma_0$. Obviously,
 $T(s)\Sigma=\Sigma$, $\forall s\ge 0$ and the following
 translation identity holds,
\begin{equation}\label{d:processId}
U_{\sigma}(t+s,\tau+s)=U_{T(s)\sigma}(t,\tau),\qquad \forall\, \sigma\in
\Sigma,\ t\ge\tau,\ \tau\in\mathbb{R},\ s\ge 0,
\end{equation}
where $\{T(s)\}_{s\ge 0}$ is the translation semigroup.  We now
consider the family of processes $\{U_\sigma(t,\tau)\}$, $\sigma\in
\Sigma$. Note that
(\ref{d:initial absorbing}) is equivalent to that for any $\tau\in
\mathbb R$ and bounded set $A\subset H$ there exists
$t_0=t_0(A)\geq\tau$, such that
\begin{equation}\label{d: sigma0 absorbing}\bigcup_{\sigma\in \Sigma}\bigcup_{t\geq
t_0}U_{\sigma}(t+\tau,\tau)A\subset B,
\end{equation}
i.e., $B\subset H$ is also a uniformly (w.r.t. $\sigma\in
\Sigma$) absorbing ball for the family of processes
$\{U_\sigma(t,\tau)\}$, $\sigma\in \Sigma$. Similarly,
(\ref{d:initial attractor}) is equivalent to that for any $\tau\in
\mathbb R$, $A\subset B$ and $\epsilon>0$ there exists
$t_0=t_0(\epsilon,A)$, such that
\begin{equation}\label{d:sigma0 attractor}
\bigcup_{\sigma\in \Sigma}\bigcup_{t\geq
t_0}U_{\sigma_0}(t+\tau,\tau)A\subset B_{H}(\mathcal A, \epsilon),
\end{equation}
i.e, $\mathcal A$  is also a uniform (w.r.t. $\sigma\in \Sigma$)
attractor for the family of processes $\{U_\sigma(t,\tau)\}$,
$\sigma\in \Sigma$. Now take $X=B$. Note that since $H$ is a separable
reflexive Banach space, both the strong and the weak topologies on
$X$ are metrizable. Define the maps $\mathcal E_\sigma$, $\sigma\in
\Sigma$ in the following way:
$$\mathcal E_{\sigma} ([\tau,\infty)) := \{u(\cdot):
u(t)=U_{\sigma}(t,\tau)u_\tau, u_\tau \in X, t\ge \tau
\}.$$
 Conditions 1--4 in the  definition of the nonautonomous evolutionary system $\mathcal E_{\sigma}$, $\sigma\in \Sigma$ follow from the definition of the family of processes
$\{U_\sigma(t,\tau)\}$, $\sigma \in \Sigma$. In addition, by
(\ref{solutiontoprocess}), we have
$$R_\sigma(t,\tau)A=U_\sigma (t,\tau)A,\quad\forall\, A\subset
X, \, \sigma\in \Sigma, \, t\ge\tau, \,\tau\in \mathbb R.$$ Hence,
the process $\{U_{\sigma_0}(t,\tau)\}$ always defines a
nonautonomous  evolutionary system  with the symbol space being the
translation family $\Sigma$ of $\sigma_0$.

In the theory of Chepyzhov and Vishik \cite{CV94, CV02}, when studying the existence and other properties,
such as the invariance of the uniform (w.r.t. $\tau\in \mathbb R$) attractor of
the process $\{U_{\sigma_0}(t,\tau)\}$, one considers a family
of processes $\{U_\sigma(t,\tau)\}$, $\sigma\in \bar{\Sigma} $ with the
symbol space $\bar{\Sigma}$ being the closure of $\Sigma$ in some
appropriate topology space. Accordingly, one considers a family of
equations with symbols in the strongly compact closure of the translation family $\Sigma$ of the
original symbol $\sigma_0$ in a corresponding functional space. In
general, suppose that a family of processes $\{U_\sigma(t,\tau)\}$,
$\sigma\in \bar{\Sigma} $ satisfies the following natural translation
identity:
\begin{equation*}U_{\sigma}(t+s,\tau+s)=U_{T(s)\sigma}(t,\tau),\qquad
\forall\, \sigma\in \bar{\Sigma},\ t\ge\tau,\ \tau\in\mathbb{R},\ s\ge 0,
\end{equation*}
and $T(s)\bar{\Sigma}=\bar{\Sigma}$, $\forall\,s\geq 0$. Proceeding in a similar manner with
$\bar{\Sigma}$ replacing $\Sigma$,  it is easy to check
that the family of processes $\{U_\sigma(t,\tau)\}$,
$\sigma\in \bar{\Sigma} $ also defines a nonautonomous  evolutionary
system with symbol space $\bar{\Sigma}$.

\subsection{Reducing a nonautonomous evolutionary  system   to an
autonomous evolutionary  system}

In this subsection we show that any nonautonomous evolutionary
system can be viewed as an (autonomous) evolutionary system.
We start with the following key lemma.
\begin{lemma}\label{reducing initial value lemma }
Let $\tau_0\in \Bbb R$ be fixed. Then for any $\tau\in \Bbb R$ and
$\sigma\in \Sigma$, there exists at least one $\sigma'\in \Sigma$
such that
\begin{equation}\label{differentsymbol}\mathcal
E_\sigma([\tau,\infty))=\{u(\cdot): u(\cdot+\tau-\tau_0)\in
\mathcal E_{\sigma'}([\tau_0,\infty))\}.\end{equation}
\end{lemma}
\begin{proof} i). Case $\tau\ge \tau_0$. Thanks to condition 2 in the
definition of the nonautonomous evolutionary system
we can just take $\sigma'=T(\tau-\tau_0)\sigma$.

ii). Case $\tau< \tau_0$. Since $\Sigma$ is invariant, there
exists at least one $\sigma'$ such that
$T(\tau_0-\tau)\sigma'=\sigma $. Again,  by condition 2 in the
definition of  $\mathcal E_\sigma$, $\sigma\in \Sigma$, we have
$$\mathcal E_{\sigma'}([\tau_0,\infty))=\{u(\cdot):
u(\cdot+\tau_0-\tau)\in \mathcal E_{\sigma}([\tau,\infty))\},$$
which is equivalent to (\ref{differentsymbol}).\end{proof}

\begin{remark} In many applications, the elements of the symbol space
$\Sigma $ are functions on the real line and $\{T(s)\}_{s\ge0}$ is
the translation semigroup. If the existence of  $\sigma'$ is
unique in Lemma~\ref{reducing initial value lemma }, corresponding to the backward uniqueness property of the system,
$\{T(s)\}_{s\ge0}$ can be extended to a group and condition 2 in the
definition of the nonautonomous evolutionary system is valid for
$s\in \mathbb R$.
\end{remark}

It follows from Lemma \ref{reducing initial value lemma } that
$$\bigcup_{\sigma\in \Sigma}R_\sigma(t,0)A=\bigcup_{\sigma\in \Sigma}R_\sigma(t+\tau, \tau)A,\quad \forall\, A\subset X, \tau\in \mathbb R, t\ge 0.$$
So it is convenient to denote
$$R_\Sigma(t)A:=\bigcup_{\sigma\in \Sigma} R_\sigma
(t,0)A,\quad \forall\, A\subset X, \ t\geq 0.$$ Similarly, we denote
\begin{align*}\mathcal E_\Sigma (I):=\bigcup_{\sigma\in \Sigma}\mathcal E_\sigma (I),\quad \forall\, I\in \mathcal
T.\end{align*}
Now we define an (autonomous) evolutionary system $\mathcal E$ in the following way:
$$\mathcal E (I):=\mathcal E_\Sigma (I),\quad
\forall\, I\in \mathcal T.$$
It is easy to check that  all the conditions in Definition~\ref{Dc} are satisfied.
Moreover, for this evolutionary
system we obviously have
\begin{equation}\label{reducingtoR(t)}R(t )A=R_\Sigma(t )A, \quad \forall\, A\subset X, t\ge0.
\end{equation}
Now the notions of invariance, quasi-invariance, and a global
attractor for $\mathcal{E}$ can be extended to the nonautonomous
evolutionary system $\{\mathcal{E}_\sigma\}_{\sigma \in \Sigma}$.
For instance, the global attractors for  evolutionary systems
defined by a process and a family of processes  in
Section \ref{S:nonauto} are the uniform (w.r.t. the initial time)
attractor and uniform (w.r.t. the symbol space) attractor,
respectively. The global attractor in the nonautonomous
case will be conventionally called a uniform global attractor (or
simply a global attractor). Other than that we will not distinguish
between autonomous and nonautonomous evolutionary systems and denote
an evolutionary system with a symbol space $\Sigma$ by
$\mathcal{E}_{\Sigma}$ and its attractor by $\mathcal A^{\Sigma}$ if it is necessary.

The advantage of such an approach will be clear
in the next section, where we will see that for some evolutionary systems
constructed from nonautonomous dynamical systems the associated symbol spaces are
not known.

\section{Closure of an evolutionary system}\label{s:closing}
In this section we will investigate evolutionary systems $
\mathcal E$ satisfying the following property:
\begin{itemize}
\item[A1] $\Dc([0,\infty))$ is a precompact set in
$ C([0,\infty); \Xw)$.
\end{itemize}
In addition, we will present some results for evolutionary systems satisfying these
additional properties:
\begin{itemize}
\item[A2] (Energy inequality)
Assume that $X$ is a set in some
Banach space $H$ satisfying the Radon-Riesz property
with the norm denoted by $|\cdot|$, such
that $\ds(x,y)=|x-y|$ for $x,y \in X$ and $\dw$ induces the weak
topology on $X$.
Assume also that for any $\epsilon
>0$, there exists $\delta$, such that for every $u \in
\Dc([0,\infty))$ and $t>0$,
\[
|u(t)| \leq |u(t_0)| + \epsilon,
\]
for $t_0$ a.e. in $ (t-\delta, t)$.
\item[A3] (Strong convergence
a.e.) Let $u_k \in \Dc([0,\infty))$, be such that
$u_k$ is $\mathrm d_{ C([0, T];\Xw)}$-Cauchy sequence in
$ C([0, T];\Xw)$ for some $T>0$. Then $u_k(t)$ is
$\ds$-Cauchy sequence a.e. in $[0,T]$.
\end{itemize}
Such kinds of evolutionary systems are closely related to the
concept of the uniform w.r.t. the initial time global attractor
for a nonautonomous system, initiated by Haraux. For instance, as
shown in the previous section, the process $\{U_{\sigma_0}(t,\tau)\}$
defines an evolutionary system $\Dc_{\Sigma}$ whose uniform global attractor
is the uniform w.r.t. the initial time global attractor due to Haraux.
However, instead of
condition \={A1}, $\Dc_{\Sigma}$ usually satisfies only A1.
The Chepyzhov-Vishik approach requires finding a suitable closure
$\bar{\Sigma}$ of the symbol space in some topological space.
In \cite{Lu07, CL09} open problems indicate that there may
not exist a symbol space $\bar{\Sigma}$, such that  a family
of processes $\{U_\sigma(t,\tau)\}$, $\sigma \in \bar{\Sigma}$ can  be
defined, since less restriction on the time-dependent nonlinearity there cannot guarantee the continuity of all  time-dependent nonlinearities of the symbols in  the symbol space $\bar{\Sigma}$. Later we will see that even for evolutionary systems taking
a closure of the symbol space is not always appropriate to study the uniform global
attractor.

Denote by $\Ab$ the uniform $\db$-global attractor of $\Dc$. We
will investigate the existence and the structure of  $\Ab$ using
a new method that involves taking a closure of  the evolutionary system $\Dc$. Let
\[
\bar{\Dc}([\tau,\infty)):=\overline{\Dc([\tau,\infty))}^{ C([\tau,
\infty);\Xw)},\quad\forall\,\tau\in\mathbb R.
\]
It can be
checked that $\bar{\Dc}$ is also an evolutionary system. We call
$\bar{\Dc}$ the closure of the evolutionary system $\Dc$, and add the top-script $\bar{\ }$ to the corresponding
notations in previous subsections  for  $\bar{\Dc}$. For instance, we denote by $\bar{\mathcal
A}_{\bullet}$ the uniform $\db$-global attractor for $\bar{\Dc}$.

First, we clearly have the following:
\begin{lemma} \label{l:eqA01A1}
If $\Dc $ satisfies A1, then $\bar{\Dc}$ satisfies \={A1}.
\end{lemma}

Now we will obtain the structure of the weak global attractor $\Aw$ as well as the weak $\omega$-limit of any weakly open set:

\begin{theorem} \label{t:baseclosing}
Let $\Dc$ be an evolutionary system satisfying A1. Then
$\ww(A)=\bar{\omega}_{\mathrm
w}(A)$ for any weakly open set $A$ in $X$. In particular,
\[
\Aw=\bar{\mathcal A}_{\mathrm w} =\{ u_0: \ u_0=u(0) \mbox{ for some } u \in
\bar{\Dc}((-\infty, \infty))\}.
\]
\end{theorem}
\begin{proof}From the definition of $\bar{\Dc}$ it follows
that
\[\overline{\bigcup_{t\geq T}R(t)A}^{\w}\subset\overline{\bigcup_{t\geq T}\bar{R}(t)A}^{\w},
\quad \forall\, T\geq 0.\] Hence, $\ww(A)\subset
\bar{\omega}_{\mathrm w}(A)$.

Now take any $x\in \bar{\omega}_{\mathrm w}(A)$. There exist
sequences $t_n\rightarrow\infty$ as $n\rightarrow \infty$ and $x_n\in
\bar{R}(t_n)A$, such that $x_n\rightarrow x$ in $\dw$-metric as
$n\rightarrow \infty$. By definition of $\bar{\Dc}$ there exist
$y_n\in R(t_n)A$ satisfying
\[
\dw(y_n,x_n)\leq \frac{1}{n}.
\]
Therefore,
\begin{equation}\label{i:eqwsww}
\dw(y_n,x)\leq\dw(y_n,x_n)+\dw(x_n,x)\leq
\frac{1}{n}+\dw(x_n,x)\rightarrow 0, \quad \mbox{as } n\rightarrow
\infty,
\end{equation}
which means that $x\in \ww(A)$. Hence, $\bar{\omega}_{\mathrm
w}(A)\subset\ww(A) $. This concludes the first part of the proof.

The second part of the theorem follows from Theorem~\ref{t:weakA} and the facts that the weak
global attractors $\Aw$ and $\bar{\mathcal A}_{\mathrm w}$ exist
and equal to $\ww(X)$ and $\bar{\omega}_{\mathrm w}(X)$,
respectively.
\end{proof}

If $\Dc$ is asymptotically compact then Theorem~\ref{t:strongA} immediately
implies that
the strong uniform global attractor $\As$ exists and $\As=\Aw$. It is also easy
to see that the strong attracting property in Theorem~\ref{t:strongA} holds
under the weaker assumption A1:

\begin{theorem}[Strong uniform tracking property]\label{t:Dcasycomstru}
Let $\Dc$ be an asymptotically compact evolutionary system
satisfying A1. Let $\bar{\Dc}$ be the closure of the  evolutionary
system $\Dc$. Then for any
$\epsilon>0$ and $T>0$, there exists $t_0$, such that for any
$t^*>t_0$, every trajectory $u \in \Dc([0,\infty))$ satisfies
\[
\ds(u(t), v(t)) < \epsilon, \quad\forall t\in [t^*,t^*+T],
\]
for some complete trajectory $v \in \bar{\Dc}((-\infty,\infty))$.
\end{theorem}
\begin{proof}
Suppose to the contrary that there
exist $\epsilon>0$, $T>0$, and sequences $u_n\in \Dc([0,\infty))$,
$t_n\rightarrow \infty$ as $n\rightarrow\infty$, such that
\begin{equation}\label{i:Dc0strongtra}
\sup_{t\in [t_n,\, t_n+T]}\ds(u_n(t), v(t)) \geq \epsilon,
\quad\forall n,
\end{equation}
for all $v \in \bar{\Dc}((-\infty,\infty))$.

On the other hand, since $\bar{\Dc}$ satisfies \={A1}, the weak
uniform tracking property in Theorem~\ref{t:weakA} implies that
there exists a sequence $v_n \in \bar{\Dc}((-\infty,\infty)) $, such
that
\begin{equation}\label{i:Dc0weaktra}
 \lim_{n\rightarrow \infty} \sup_{t\in [t_n,\, t_n+T]}\dw(u_n(t),
 v_n(t))=0.
 \end{equation}
 Thanks to (\ref{i:Dc0strongtra}), there exists a sequence
 $\hat{t}_n\in [t_n,t_n+T]$, such that
\begin{equation}\label{i:Dc0strongtratn}
\ds(u_n(\hat{t}_n), v_n(\hat{t}_n)) \geq \epsilon/2, \quad\forall n,
\end{equation}

Now note that $\{u_n(\hat{t}_n)\}$ is relatively strongly compact due
to the asymptotic compactness of $\Dc$. In addition, Theorem~\ref{t:baseclosing}
implies that
\[
\{v_n(\hat{t}_n)\} \subset \Aw.
\]
Thanks again to the asymptotic compactness of $\Dc$, $\Aw$ is strongly
compact due to Theorem~\ref{t:strongA}. Hence, the sequence
$\{v_n(\hat{t}_n)\}$ is also relatively strongly compact.
Then it follows from (\ref{i:Dc0weaktra})
that the limits of the convergent subsequences of
$\{u_n(\hat{t}_n)\}$ and $\{v_n(\hat{t}_n)\}$ coincide, which
contradicts (\ref{i:Dc0strongtratn}).
\end{proof}

Finally, in order to extend Theorem~\ref{t:AComp} to $\Dc$ we need the following:
\begin{lemma}\label{l:eqA02A03A2A3}
If $\Dc$ satisfies A2 and A3, then $\bar{\Dc}$ satisfies \={A2} and \={A3}.
\end{lemma}
\begin{proof}
Clearly \={A3} holds by definition of $\bar{\Dc}$.
Now take $u \in \bar{\Dc}([0,\infty))$ and $T>0$. There exists a sequence $u_n \in
\Dc([0,\infty))$ satisfying
\[
u_n\rightarrow u \mbox{ in } C([0,T];\Xw).
\]
Thanks to \={A3},
\[
u_n(t)\rightarrow u(t) \mbox{ strongly in } [0,T]\setminus E_0,
\]
where $E_0$ is a set of zero measure.
Due to A2, for any $\epsilon
>0$, there exists $\delta$, such that for every $u_n \in
\Dc([0,\infty))$ and $T>0$,
\[
|u_n(T)| \leq |u_n(t)| + \epsilon,
\]
for $t$ in $ (T-\delta, T)\setminus E_n$, where $E_n$ is a  zero measure set.
Taking the lower limit as $n\rightarrow \infty$ we obtain
\[
|u(T)| \leq\liminf |u_n(T)| \leq |u(t)| + \epsilon, \quad t\in (T-\delta, T)\setminus \cup_{i=0}^{\infty}E_i,
\]
which means that \={A2} holds.
\end{proof}

With the above results in hand, we conclude with the following versions of
Theorems~\ref{t:weakA}, \ref{t:strongA}, and \ref{t:AComp} for $\Dc$.
\begin{theorem} \label{t:weakA0}
Let $\Dc$ be an evolutionary system. Then \begin{itemize}\item
[1.] The weak global attractor $\Aw$ exists.\end{itemize}
Furthermore, assume that  $\Dc$ satisfies A1. Let $\bar{\Dc}$ be
the closure of $\Dc$. Then
\begin{itemize}
\item [2.]$ \Aw =\ww(X)=\bar{\omega}_{\mathrm w}(X)=\bar{\omega}_{\mathrm s}(X)=\bar{\mathcal A}_{\mathrm w}=\{ u_0 \in X:  u_0=u(0) \mbox{ for some } u \in
\bar{\Dc}((-\infty, \infty))\} $. \item [3.]$\Aw$ is the maximal
invariant and maximal quasi-invariant set w.r.t. $\bar{\Dc}$. \item
[4.](Weak uniform tracking property) For any $\epsilon >0$, there
exists $t_0$, such that for any $t^*>t_0$, every trajectory $u \in
\Dc([0,\infty))$ satisfies $ \dd_{C([t^*,\infty);\Xw)} (u, v) <
\epsilon, $ for some complete trajectory $v \in
\bar{\Dc}((-\infty,\infty))$.
\end{itemize}
\end{theorem}

\begin{theorem} \label{t:strongA0}
Let $\Dc$ be an asymptotically compact evolutionary system.
Then
\begin{enumerate}
\item[1.]
The strong global attractor $\As$ exists, it is strongly compact, and
$\As=\Aw$.
\end{enumerate}
Furthermore, assume that $\Dc$ satisfies A1. Let $\bar{\Dc}$ be
the closure of $\Dc$. Then
\begin{enumerate}
\item[2.](Strong uniform tracking
property)
For any $\epsilon >0$ and $T>0$, there exists $t_0$, such
that for any $t^*>t_0$, every trajectory $u \in \Dc([0,\infty))$
satisfies $\ds(u(t), v(t)) < \epsilon$, $\forall t\in [t^*,t^*+T]$,
for some complete trajectory $v \in \bar{\Dc}((-\infty,\infty))$.
\end{enumerate}
\end{theorem}

\begin{theorem} \label{t:AComp0}
Let $\Dc$ be an evolutionary system satisfying A1, A2, and
A3, and assume that its closure $\bar{\Dc}$ satisfies
$\bar{\Dc}((-\infty,\infty))\subset C((-\infty, \infty);\Xs)$. Then
$\Dc$ is asymptotically compact.
\end{theorem}

\subsection{Uniform w.r.t. symbol space global attractors}\label{s:subES}
In \cite{CV94, CV02}, Chepyzhov and Vishik studied the structure of
the uniform (w.r.t. the initial time) global attractor of a process via
that of the uniform (w.r.t. the symbol space) attractor of a family
of processes with the symbol space being the strong closure of the translation family of the
original symbol in an appropriate functional space. For further
results in the case where the symbol space is a weak closure of the translation family of the original symbol we refer to \cite{LWZ05,
Lu06, Lu07}. In some cases (see e.g. open problems
in \cite{Lu07, CL09}) it is not clear how to choose a symbol space to obtain
the structure of the uniform (w.r.t. the initial time) global attractor.
Even though we solved this problem in Section~\ref{s:closing} using a different
approach,
the uniform (w.r.t. the symbol space) attractor remains  of mathematical
interest. In this subsection we study this object and its relation to the uniform (w.r.t. the initial time) global attractor using our framework of evolutionary system.

\begin{definition}
Let $\Dc$ be an evolutionary system. If a map $\Dc^1$ that
associates to each $I\in \mathcal T$ a subset $\Dc^1(I)\subset
\Dc(I)$ is also an evolutionary system, we will call it an
evolutionary subsystem of $\Dc$, and denote by $\Dc^1\subset \Dc$.
\end{definition}

Let $\Dc_{\bar{\Sigma}}$ be an evolutionary system, and let $\Sigma\subset
\bar{\Sigma}$ be such that $T(h)\Sigma=\Sigma$ for all $h\geq0$. Then
it is easy to check that $\Dc_{\Sigma}$ is also an evolutionary
system. Hence, it is  an evolutionary subsystem of $\Dc_{\bar{\Sigma}}$. For
example, in Section \ref{S:nonauto}, the evolutionary system defined
by a process $\{U_{\sigma_0}(t,\tau)\}$ is an evolutionary subsystem
of the evolutionary system defined by the family of processes
$\{U_\sigma(t,\tau)\}$, $\sigma\in \bar{\Sigma}$, where $\bar{\Sigma}$ is the
closure of the translation family $\Sigma$ of the symbol
$\sigma_0$ in some appropriate topological space.

\begin{definition} An evolutionary system $ \mathcal E_\Sigma$ is a system
with uniqueness if
for every $u_0\in X$ and $\sigma\in \Sigma$, there is a unique
trajectory $u\in \mathcal E_{\sigma}([0,\infty))$ such that
$u(0)=u_0$.
\end{definition}

Examples of  evolutionary systems with uniqueness include the evolutionary systems defined previously by  a process and a
family of processes.

\begin{theorem}\label{t:A0Aw}
Let $\Dc_\Sigma$ be an evolutionary system with uniqueness and  with symbol space $\Sigma$ satisfying A1. Let $\bar{\Sigma}$ be the closure of $\Sigma$ in some topological space
$\Im$ and $\Dc_{\bar{\Sigma}}\supset\Dc_\Sigma$ be an evolutionary system with uniqueness
satisfying \={A1}, and such that   $u_n\in \Dc_{\sigma_n}([0,\infty))$,
$u_n\rightarrow u$ in $C([0, \infty);\Xw)$ and $\sigma_n\rightarrow
\sigma$ in $\Im$ imply $u\in \Dc_{\sigma}([0,\infty))$. Then, their
weak uniform global attractors $\Aw^\Sigma$ and $\mathcal A^{\bar{\Sigma}}_{\mathrm{w}}$ are identical.
\end{theorem}
\begin{proof}
Obviously, $\Aw^\Sigma$ and $\mathcal A^{\bar{\Sigma}}_{\mathrm{w}}$ exist and $\Aw^\Sigma\subset \mathcal A^{\bar{\Sigma}}_{\mathrm{w}}$. If there
exists $x_0\in \mathcal A^{\bar{\Sigma}}_{\mathrm{w}}\setminus\Aw^\Sigma$, then  there exist two disjoin
balls $B_{\mathrm w}(\Aw^\Sigma, \epsilon)$ and $B_{\mathrm w}(x_0,
\epsilon)$. Since $$\{u(t)|t\in \mathbb R,
u\in\Dc_\Sigma((-\infty,\infty))\}\subset\Aw^\Sigma,$$ we can take, by Theorem
\ref{t:weakA},  a complete trajectory $v(t)\in
\Dc_\sigma((-\infty,\infty))$ with $\sigma\in
\bar{\Sigma}\setminus\Sigma$ such that $v(0)=x_0$. The set $\{v(t)|t\in
\mathbb R\}$ is $\dw$-attracted by $\Aw^\Sigma$. Hence, there is some
$t_0$ such that,
\begin{equation*}
R_{\Sigma}(t_0)\{v(t)|t\in\mathbb R\}\subset B_{\mathrm w}(\Aw^\Sigma,
\epsilon).
\end{equation*}
Note that $\bar{\Sigma}$ is the closure of $\Sigma$ in $\Im$. Take a
sequence $\sigma_n\in \Sigma$ such that
$\sigma_n\rightarrow\sigma$ in $\Im$. Consider a sequence of trajectories
$u_n(t)\in \Dc_{\sigma_n}([0,\infty))$ satisfying $u_n(0)=v(-t_0)$.
We have
\begin{equation}\label{i:unaw0}
u_n(t_0)\in B_{\mathrm w}(\Aw^\Sigma, \epsilon).
\end{equation}
Thanks to \={A1}, $\{u_n(t)\} $ converges, passing to a subsequence and
dropping a subindex, in $C([0, \infty);\Xw)$, whose limit $u(t)\in
\Dc_\sigma([0,\infty))$ due to $\sigma_n\rightarrow\sigma$ in $\Im$.
By the fact that $u(0)=v(-t_0)$ and the uniqueness of the
evolutionary system $\Dc_{\bar\Sigma}$, $u(t)=v(t-t_0)$, $t\geq 0$.
However, (\ref{i:unaw0}) indicates that
$$x_0=v(0)=u(t_0)\in \overline{B_{\mathrm w}(\Aw^\Sigma, \epsilon)}^{\mathrm
w}.$$This is a contradiction. Hence, $\Aw^\Sigma=\mathcal A^{\bar{\Sigma}}_{\mathrm{w}}$.
\end{proof}

Therefore, together with Theorems \ref{t:weakA0} and \ref{t:weakA},
Theorem \ref{t:A0Aw} implies the following:
\begin{theorem}\label{t:A0AwA-weak}
Under the conditions of Theorem \ref{t:A0Aw}, let $\bar{\Dc}_\Sigma$ be the
closure of the evolutionary system $\Dc_\Sigma$. Then the three weak uniform global
attractors $\Aw^\Sigma$, $\bar{\mathcal A}^\Sigma_{\mathrm w}$ and $\mathcal A^{\bar{\Sigma}}_{\mathrm{w}}$ of the
evolutionary systems $\Dc_\Sigma$, $\bar{\Dc}_\Sigma$ and $\Dc_{\bar{\Sigma}}$,
respectively, are identical, and the following invariance property
holds
\begin{equation}\label{i:A0unistru}
\begin{split}
\Aw^\Sigma&=\bar{\mathcal A}^\Sigma_{\mathrm
w}=\mathcal A^{\bar{\Sigma}}_{\mathrm{w}}\\
&=\{ u_0: \ u_0=u(0) \mbox{ for some } u \in \bar{\Dc}_\Sigma((-\infty,
\infty))\}\\
&=\{ u_0: \ u_0=u(0) \mbox{ for some } u \in \Dc_{\bar{\Sigma}}((-\infty,
\infty))\}.
\end{split}
\end{equation}
Moreover, the weak uniform tracking property holds.
\end{theorem}
Now, Theorem \ref{t:strongA} ensures the strong compactness.

\begin{theorem}\label{t:A0AwA-strong1}Under the conditions of Theorem \ref{t:A0AwA-weak}, assume that
 $\Dc_{\bar{\Sigma}}$ is asymptotically compact. Then the
weak uniform global attractors in Theorem \ref{t:A0AwA-weak} are
strongly compact strong uniform global attractors. Moreover, the
strong uniform tracking property holds.
\end{theorem}
In applications the auxiliary  evolutionary system
$\Dc_{\bar{\Sigma}}$ is usually asymptotically compact. For instance, in \cite{Lu06}, in the case of the 2D Navier-Stokes
equations with non-slip boundary condition, $\bar{\Sigma}$ is taken as the
closure of the translation family  of a normal external force (see Section \ref{3DNSE}) in
$L^{2,\mathrm{w}}_{\mathrm {loc}}(\mathbb R;V')$. Here, $V'$ is the
dual of the space of divergence-free vector fields with
square-integrable derivatives and vanishing on the boundary, and
$L^{2,\mathrm{w}}_{\mathrm {loc}}(\mathbb R;V')$ is the space
$L^{2}_{\mathrm {loc}}(\mathbb R;V')$ endowed with local weak
convergence topology. Then Theorem~\ref{t:A0AwA-strong1} applied to this system
gives the strong uniform tracking property.

Finally, together with Theorem \ref{t:AComp}, Theorem \ref{t:A0AwA-strong1} implies the following:
\begin{theorem}\label{t:A0AwA-strong}Under the conditions of Theorem \ref{t:A0AwA-weak}, assume that
$\Dc_{\bar{\Sigma}}$ satisfies \={A2}, \={A3} and every complete
trajectory in (\ref{i:A0unistru}) is strongly continuous. Then the
weak uniform global attractors in Theorem \ref{t:A0AwA-weak} are
strongly compact strong uniform global attractors. Moreover, the strong
uniform tracking property holds.
\end{theorem}

\section{Trajectory attractor}\label{s:traattra}
A trajectory attractor for the 3D NSE was introduced in \cite{Se96}
and further studied in \cite{CV97, CV02, SY02} by considering a family of auxiliary nonautonomous systems
including the original system. In this section, with the results in
preceding sections in hand, we will naturally construct a
trajectory attractor for the original system under consideration, rather than for a family of systems. More precisely, we construct a trajectory attractor
for the evolutionary system $\Dc$ satisfying A1 utilizing the
trajectory attractor for its closure $\bar{\Dc}$,
which is defined in \cite{C09}.

Let $\mathcal F^+:=C([0,\infty);\Xw)$ and denote
$$\mathcal K^+:=\Dc([0\,\infty))\subset \mathcal F^+.$$
 Define the translation operators $T(s)$, $s\geq0$,
$$(T(s)u)(t):=u(t+s)|_{[0,\infty)},\quad u\in \mathcal F^+.$$
Due to the property 3 of the evolutionary system (see Definitions
\ref{Dc} and \ref{d:Dc0}), we have that,
$$T(s)\mathcal K^+\subset \mathcal K^+,\quad \forall\, s\geq 0.$$
Note that $\mathcal K^+$ may not be closed, but is precompact in
$\mathcal F^+$ due to A1. For a set $P\subset\mathcal F^+$ and
$r>0$ denote
$$B(P,r):=\{u\in \mathcal F^+:\dd_{C([0,\infty);\Xw)}(u,P)<r\}.
$$
A set $P\subset\mathcal F^+$ uniformly attracts a set
$Q\subset\mathcal K^+$ if for any $\epsilon>0$ there exists $t_0$,
such that
$$T(t)Q\subset B(P,\epsilon), \quad\forall\, t\geq t_0.$$
\begin{definition}
A set $P\subset\mathcal F^+$ is a trajectory attracting set for an evolutionary system $\Dc$ if it
uniformly attracts $\mathcal K^+$.
\end{definition}
\begin{definition}
A set $\mathfrak A\subset\mathcal F^+$ is a trajectory attractor for an evolutionary system $\Dc$ if $\mathfrak
A$ is a minimal compact trajectory attracting set, and
$T(t)\mathfrak A=\mathfrak A$ for all $t\geq 0$.
\end{definition}
It is easy to see that if a trajectory attractor exists, it is unique. Let
$\bar{\Dc}$ be the closure of the evolutionary system $\Dc$ and let
$\bar{\mathcal K}:=\bar{\Dc}((-\infty, \infty))$ which is called the
kernel of $\bar{\Dc}$. Let also
$$\Pi_+\bar{\mathcal K}:=\{u(\cdot)|_{[0,\infty)}:u\in \bar{\mathcal K}\}. $$

\begin{theorem}\label{t:weaktrattracor}
Let $\Dc$ be an evolutionary system satisfying A1. Then the
trajectory attractor exists and
\[\mathfrak A=\Pi_+\bar{\mathcal K}, \]
where $\bar{\mathcal K}$ is the kernel of the closure $\bar{\Dc}$ of the  evolutionary system $\Dc$. Furthermore,
\[\Aw=\mathfrak A(t):=\{u(t): u\in \mathfrak A\},\quad \forall\, t\geq 0.\]
\end{theorem}
\begin{proof}
Notice that Theorem 7.4 in \cite{C09} states that the conclusions are valid for an evolutionary system satisfying \={A1}.

Obviously, $\bar{\Dc}$ satisfies \={A1}.
Hence the trajectory attractor $\Pi_+\bar{\mathcal K}$ for $\bar{\Dc}$ uniformly attracts $\bar{\mathcal K}^+$. Now we verify that $\Pi_+\bar{\mathcal K}$ is a minimal trajectory attracting set for $\Dc$. Assume that there exists a compact trajectory attracting set $P$ strictly included in $\Pi_+\bar{\mathcal K}$. Then there exist $\epsilon > 0 $ and
\[
u \in \Pi_+\bar{\mathcal K}\setminus B(P, 2\epsilon).
\]
Let $v\in \bar{\Dc}((-\infty,\infty)) $ be such that $v|_{[0,\infty)} = u$. Let also $v_n(\cdot) = v(\cdot-n)|_{[0,\infty)}$. Note that
$v_n\in \bar{\Dc}([0,\infty))$ and
\[
T (n)v_n = u \notin B(P, 2\epsilon), \quad \forall \ n.
\]
Now take $u_n\in \Dc([0,\infty))$ such that
\[
\dd_{C([0,\infty);\Xw)}(u_n,v_n)<\epsilon/2^n,\quad \forall \ n.
\]
By the definition of the metric $\dd_{C([0,\infty);\Xw)}$, we have
\[
\dd_{C([0,\infty);\Xw)}(T(n)u_n,T(n)v_n)
\leq 2^n\dd_{C([0,\infty);\Xw)}(u_n,v_n),\quad \forall \ n.
\]
Hence,
\[
\dd_{C([0,\infty);\Xw)}(T(n)u_n,u)
< \epsilon,\quad \forall \ n,
\]
which implies that
\[
T (n)u_n  \notin B(P, \epsilon), \quad \forall \ n.
\]
Therefore, $P $ is not a trajectory attracting set for $\Dc$, which is  a contradiction.
\end{proof}

Furthermore, the asymptotic compactness of $\Dc$
implies a uniform strong convergence of solutions toward the
trajectory attractor.
\begin{theorem}\label{t:strongtrattracorasy}
Let $\Dc$ be an asymptotically compact evolutionary system satisfying A1. Then the
trajectory attractor $\mathfrak A$ uniformly attracts $\mathcal K^+$
in $L^\infty_{\mathrm{loc}}((0,\infty);\Xs)$.
\end{theorem}
\begin{proof}
This is just a consequence of  Theorem \ref{t:Dcasycomstru}.
\end{proof}
Finally, by the  strong continuity of the complete trajectories, we have the following.
\begin{theorem}\label{t:strongtrattracor}
Let $\Dc$ be an evolutionary system satisfying A1, A2 and
A3. If $\mathfrak A\subset C([0, \infty);\Xs)$, then the
trajectory attractor $\mathfrak A$ uniformly attracts $\mathcal K^+$
in $L^\infty_{\mathrm{loc}}((0,\infty);\Xs)$.
\end{theorem}
\begin{proof}
Since $\mathfrak A\subset C([0, \infty);\Xs)$, Theorem \ref{t:AComp0} implies that the evolutionary system $\Dc$ is
asymptotically compact. Therefore, Theorem \ref{t:strongtrattracorasy} yields that $\mathfrak A$ uniformly attracts $\mathcal K^+$
in $L^\infty_{\mathrm{loc}}((0,\infty);\Xs)$.
\end{proof}

\section{3D Navier-Stokes equations}\label{3DNSE}
Consider the space periodic 3D incompressible
Navier-Stokes equations (NSE)
\begin{equation} \label{NSE1}
\left\{
\begin{aligned}
&\ddt u - \nu \Delta u + (u \cdot \nabla)u + \nabla p = f(t),\\
&\nabla \cdot u =0,
\end{aligned}
\right.
\end{equation}
where $u$, the velocity, and $p$, the pressure, are unknowns; $f(t)$
is a given driving force, and $\nu>0$ is the kinematic  viscosity
coefficient of the fluid. By a Galilean change of variables, we can
assume that the space average of $u$ is zero, i.e.,
\[
\int_\Omega u(x,t) \, dx =0, \qquad \forall t,
\]
where $\Omega=[0,L]^3$ is a periodic box.\footnote{The no-slip
case can be considered in a similar way, only with some adaption
on the functional  setting.}

First, let us introduce some notations and functional setting.
Denote by $(\cdot,\cdot)$ and $|\cdot|$ the $L^2(\Omega)^3$-inner
product and the corresponding $L^2(\Omega)^3$-norm. Let
$\mathcal{V}$ be the space of all $\mathbb{R}^3$ trigonometric
polynomials of period $L$ in each variable satisfying $\nabla \cdot
u =0$ and $\int_\Omega u(x) \, dx =0$. Let $H$ and $V$  be the
closures of $\mathcal{V}$ in $L^2(\Omega)^3$ and $H^1(\Omega)^3$,
respectively. Define the strong and weak distances by
\[
\ds(u,v):=|u-v|, \qquad
\dw(u,v)= \sum_{\kappa \in \mathbb{Z}^3} \frac{1}{2^{|\kappa|}}
\frac{|u_{\kappa}-v_{\kappa}|}{1 + |u_{\kappa}-v_{\kappa}|},
\qquad u,v \in H,
\]
where $u_{\kappa}$ and $v_{\kappa}$ are Fourier coefficients of $u$
and $v$ respectively. Note that the weak metric $\dw$ induces the
weak topology in any ball in $L^2(\Omega)^3$.

Let also  $P_{\sigma} : L^2(\Omega)^3 \to H$ be the $L^2$-orthogonal
projection, referred to as the Leray projector. Denote by
$A=-P_{\sigma}\Delta = -\Delta$ the Stokes operator with the domain
$D(A)=(H^2(\Omega))^3 \cap V$. The Stokes operator is a self-adjoint
positive operator with a compact inverse.
Let
\[
\|u\| := |A^{1/2} u|,
\]
which is called the enstrophy norm.
Note that $\|u\|$ is equivalent to the $H^1$-norm of $u$ for $u\in D(A^{1/2})$.

Now denote $B(u,v):=P_{\sigma}(u \cdot \nabla v)\in V'$ for all
$u, v \in V$. This bilinear form has the following property:
\[
\langle B(u,v),w\rangle=-\langle B(u,w),v\rangle, \qquad u,v,w \in V,
\]
in particular, $\langle B(u,v),v\rangle=0$ for all $u,v \in V$.

Now we can rewrite (\ref{NSE1}) as the following differential equation in $V'$:

\begin{equation} \label{NSE}
\ddt u + \nu A u +B(u,u) = g,\\
\end{equation}
where $u$ is a $V$-valued function of time and $g = P_{\sigma} f$.

\begin{definition}
A weak solution  of  \eqref{NSE1} on $[T,\infty)$ (or $(-\infty, \infty)$, if
$T=-\infty$) is an $H$-valued
function $u(t)$ defined for $t \in [T, \infty)$, such that
\[
\ddt u \in L_{\mathrm{loc}}^1([T, \infty); V'), \qquad
u(t) \in C([T, \infty); \Hw) \cap
L_{\mathrm{loc}}^2([T, \infty); V),
\]
and
\begin{equation}\label{IntNSE}
\left(u(t)-u(t_0), v\right) = \int_{t_0}^t \left( -\nu ((u, v)) - \langle B(u,u),v\rangle +\langle g, v\rangle \right) \, ds,
\end{equation}
for all $v \in V$ and  $T \leq t_0 \leq t $.
\end{definition}

\begin{theorem}[Leray, Hopf] \label{thm:Leray}
For every $u_0 \in H$ and $g \in L^2_\mathrm{loc}(\mathbb{R};V')$, there exists a
weak solution of (\ref{NSE1}) on $[T,\infty)$ with $u(T)=u_0$
satisfying the following energy inequality
\begin{equation} \label{EI}
|u(t)|^2 + 2\nu \int_{t_0}^t \|u(s)\|^2 \, ds \leq
|u(t_0)|^2 + 2\int_{t_0}^t \langle g(s), u(s)\rangle \, ds
\end{equation}
for all $t \geq t_0$, $t_0$ a.e. in $[T,\infty)$.
\end{theorem}

\begin{definition} \label{d:ex}
A Leray-Hopf solution of \eqref{NSE1} on the interval $[T, \infty)$
is a weak solution on $[T,\infty)$ satisfying the
energy inequality (\ref{EI}) for all $T \leq t_0 \leq t$,
$t_0$ a.e. in $[T,\infty)$. The set $Ex$ of measure $0$ on which the energy
inequality does not hold will be called the exceptional set.
\end{definition}

Now fix an external  force $g_0$  that is   translation bounded in
$L^2_\mathrm{loc}(\mathbb{R};V')$ , i.e.,
\[
\|g_0\|^2_{L^2_{\mathrm{b}}} := \sup_{t \in \mathbb{R}} \int_t^{t+1}
\|g_0(s)\|_{V'}^2 \, ds < \infty.
\]
Then $g_0$ is translation compact in
$L^{2,\mathrm{w}}_\mathrm{loc}(\mathbb{R};V')$, i.e., the
translation family of $g_0$
\[
\Sigma:=\{g_0(\cdot+h)|h\in \mathbb R\}
\]
is precompact in
$L^{2,\mathrm{w}}_\mathrm{loc}(\mathbb{R};V')$.  Note that,
\begin{equation}\label{i:g}\|g\|^2_{L^2_{\mathrm{b}}}\leq
\|g_0\|^2_{L^2_{\mathrm{b}}},\quad \forall\, g\in
\Sigma.\end{equation}Due to the energy inequality \eqref{EI} we have
\[
|u(t)|^2 + \nu \int_{t_0}^t \|u(s)\|^2 \, ds \leq |u(t_0)|^2 +
\frac{1}{\nu}\int_{t_0}^t \|g(s)\|^2_{V'} \, ds, \qquad \forall g
\in \Sigma,
\]
for all $t \geq t_0$, $t_0$ a.e. in $[T,\infty)$. Here $u(t)$ is a Leray-Hopf solutions of  \eqref{NSE1} with force
$g$ on $[T,\infty)$.  By Gronwall's inequality there
exists an absorbing ball $B_{\mathrm{s}}(0, R)$, where the radius
$R$ depends on $L$, $\nu$, and $\|g_0\|^2_{L^2_{\mathrm{b}}}$.
 Let $X$ be a closed absorbing ball
\[
X= \{u\in H: |u| \leq R\},
\]
which is also weakly compact. Then for any bounded set $A \subset
H$, there exists a time $t_1\ge T$, such that
\[
u(t) \in X, \qquad \forall t\geq t_1,
\]
for every Leray-Hopf solution $u(t)$ with the force $g\in \Sigma$
and the initial data $u(T) \in A$.  For any sequence of
Leray--Hopf solutions $u_n$ the following result holds.

\begin{lemma} \label{l:precompactofLH}
Let $u_n(t)$ be a sequence of Leray-Hopf solutions of  \eqref{NSE1} with forces
$g_n \in \Sigma$, such that $u_n(t) \in X$ for all $t\geq t_1$. Then
\begin{equation}\label{bdofun}
\begin{split}
u_n \ \ &\mbox{is bounded in} \ \ L^2(t_1,t_2;V)\ \ \mbox{and} \ \ L^\infty(t_1,t_2;H),\\
\ddt u_n \ \  &\mbox{is bounded in} \ \ L^{4/3}(t_1,t_2;V'),
\end{split}
\end{equation}
for all $t_2>t_1$. Moreover, there exists a subsequence $u_{n_j}$ converges to some $u(t)$ in $C([t_1, t_2]; \Hw)$,
i.e.,
\[
(u_{n_j},v) \to (u,v) \qquad
\mbox{uniformly on} \qquad  [t_1,t_2],
\]
as $n_j\to \infty$, for all $v \in H$.
\end{lemma}

\begin{proof}The proof is standard (see e.g. \cite{CF89, Ro01}). Here we just sketch some steps. Take a sequence $u_n$ satisfying \eqref{NSE1} with forces $g_n$. By \eqref{NSE}, we have
\begin{equation} \label{NSEn}
\ddt u_n + \nu A u_n +B(u_n,u_n) = g_n,\\
\end{equation}
Classical estimates imply the boundedness in \eqref{bdofun}. Then, passing to  a subsequence and dropping a subindex, we can obtain that
\[
\begin{aligned}
u_{n}\rightarrow u \quad &\text{weak-star in }L^\infty(t_1,t_2;H),\\
&\text{weakly in }L^2(t_1,t_2;V),\\
&\text{strongly in }L^2(t_1,t_2;H),
\end{aligned}
\]
  and
 \[
\begin{aligned} \ddt u_{n}\rightarrow \ddt u \quad &\text{weakly in
}L^{4/3}(t_1,t_2;V'),\\
Au_{n}\rightarrow Au\quad &\text{weakly in }L^2(t_1,t_2;V'),\\
B(u_n,u_n)\rightarrow B(u,u)\quad &\text{weakly in }L^{4/3}\left(t_1,t_2;V'\right),
\end{aligned}
\]
 for some
\[
u\in L^\infty(t_1,t_2;H)\cap L^2(t_1,t_2;V).
\]
Again, passing to  a subsequence and dropping a subindex, we also have,
\begin{equation}\label{convofgn}
g_n\rightarrow g\quad \text{ weakly in
}L^{2}(t_1,t_2;V'),\end{equation}with $g\in L^{2}(t_1,t_2;V')$.
Passing to the limit  yields
\[
\begin{aligned}\label{quasiequ} \ddt u + \nu A u +B(u,u) = g.
\end{aligned}
\]
It follows from \eqref{IntNSE} that $u_{n}\rightarrow u$ in $C([t_1,t_2]; H_{\mathrm w})$.
\end{proof}

\begin{remark} \label{r:limitLH}
In the autonomous case, i.e., $f(t)$ is independent of $t$, the limit $u$ is a Leray-Hopf solution. However, we don't know here whether it is a Leray-Hopf solution yet.
\end{remark}
Consider an evolutionary system for which a family of trajectories
consists of all Leray-Hopf solutions of the 3D Navier-Stokes
equations with a fixed force $g_0$ in $X$. More precisely, define
\[
\begin{split}
\Dc([T,\infty)) := \{&u(\cdot): u(\cdot)
\mbox{ is a Leray-Hopf}
\mbox{ solution on } [T,\infty)\\
&\mbox{with the force } g\in \Sigma \mbox{ and } u(t) \in X, \
\forall t \in [T,\infty)\}, \   T \in \mathbb{R},
\end{split}
\]
\[
\begin{split}
\Dc((-\infty,\infty)) := \{&u(\cdot): u(\cdot)
\mbox{ is a Leray-Hopf} \ \mbox{ solution on } (-\infty,\infty)\\
&\mbox{with the force } g\in \Sigma \mbox{ and } u(t) \in X, \
\forall t \in (-\infty,\infty)\}.
\end{split}
\]

Clearly, the properties 1--4 of $\Dc$ hold, if we utilize the
translation semigroup $\{T(s)\}_{s\ge 0}$. Therefore, thanks to
Theorem~\ref{t:weakA0}, the uniform  weak global attractor $\Aw$ for this
evolutionary system exists.

 Now  we give the definition of normal function which was first put forward
in \cite{LWZ05}.
\begin{definition}\label{d:normal} Let $\mathcal B$ be a Banach space. A function $\varphi (s)\in L^2_{\mathrm{loc}}(\mathbb{R};\mathcal B
)$ is said to be normal in $L^2_{\mathrm{loc}}(\mathbb{R};\mathcal B
)$ if for any $\epsilon>0$, there exists $\delta>0$, such that
\[
\sup_{t\in \mathbb{R}} \int_{t}^{t+\delta}
\|\varphi(s)\|^2_{\mathcal B} \, ds \leq  \epsilon.
\]
\end{definition}Note that the class of normal
functions is a proper closed subspace of
the class of translation bounded functions (see \cite{LWZ05} for more details). Then, we have the
following.
\begin{lemma} \label{l:compact0}
The evolutionary system $\Dc$ of the 3D NSE with the force $g_0$
satisfies A1 and A3. Moreover, if $g_0$ is normal in
$L^2_{\mathrm{loc}}(\mathbb{R}; V')$ then A2 holds.
\end{lemma}
\begin{proof}
First note that $\Dc([0,\infty)) \subset C([0,\infty);\Hw)$
by the definition of a Leray-Hopf solution. Now take any sequence
$u_n \in \Dc([0,\infty))$, $n=1,2, \dots$. Thanks to
Lemma~\ref{l:precompactofLH}, there exists a subsequence, still
denoted by  $u_n$, that converges to some $u^{1} \in
C([0, 1];\Hw)$ in $C([0, 1];\Hw)$ as $n \to \infty$.
Passing to a subsequence and dropping a subindex once more, we
obtain that $u_n \to u^2$ in $C([0, 2];\Hw)$ as $n \to \infty$ for
some $u^{2} \in C([0, 2];\Hw)$. Note that $u^1(t)=u^2(t)$
on $[0, 1]$. Continuing this diagonalization process, we obtain a
subsequence $u_{n_j}$ of $u_n$ that converges to some $u \in
C([0, \infty);\Hw)$ in $C([0, \infty);\Hw)$ as $n_j \to \infty$.
Therefore, A1 holds.

Let now $u_n \in \Dc([0,\infty))$ be such that $u_n \to
u\in C([0, T];\Hw)$ in $C([0, T];\Hw)$ as $n\to \infty$ for
some $T>0$. Thanks to Lemma~\ref{l:precompactofLH} again, the sequence
$\{u_n\}$ is bounded in $L^2([0,T];V)$. Hence,
\[
\int_{0}^T |u_n(s)-u(s)|^2 \, ds \to 0, \qquad \mbox{as}
\qquad  n \to \infty.
\]
In particular, $|u_n(t)| \to |u(t)|$ as $n \to \infty$ a.e. on $[0,T]$,
i.e., A3 holds.

Now assume that $g_0$ is normal in $L^2_{\mathrm{loc}}(\mathbb{R};
V')$.
Then given $\epsilon>0$, there exists $\delta>0$, such that
\[
\sup_{t\in \mathbb{R}} \int_{t-\delta}^t \|g_0(s)\|^2_{V'} \, ds
\leq \nu \epsilon.
\]
Take any $u \in \Dc([0,\infty))$ and $t>0$. Since $u(t)$ is a
Leray-Hopf solution, it satisfies the energy inequality \eqref{EI}
\[
|u(t)|^2 + 2\nu \int_{t_0}^t \|u(s)\|^2 \, ds \leq
|u(t_0)|^2 + 2\int_{t_0}^t \langle g(s), u(s)\rangle \, ds,
\]
for all $0 \leq t_0 \leq t$, $t_0 \in [0,\infty) \setminus Ex$,
where $Ex$ is a set of zero measure. Hence, together with
(\ref{i:g}),
\[
\begin{split}
|u(t)|^2  &\leq |u(t_0)|^2 + \frac{1}{\nu}\int_{t_0}^t
\|g_0\|^2_{V'}\, ds\\ &\leq |u(t_0)|^2 + \epsilon,
\end{split}
\]
for all $t_0\geq 0$, such that $t_0 \in (t-\delta,t)\setminus Ex$.
Therefore, A2 holds.
\end{proof}

Now Lemma \ref{l:compact0}, Theorem \ref{t:weakA0}, \ref{t:strongA0} and \ref{t:AComp0} yield the following.
\begin{theorem}\label{t:Aw0NSE} The uniform weak global attractor $\Aw$ for the 3D
NSE with force $g_0$ exists, $\Aw$ is the maximal invariant and
maximal quasi-invariant set w.r.t. the closure $\bar{\Dc}$ of the corresponding evolutionary system $\Dc$, and
\[
\Aw = \ww(X)=\ws(X)=\{u(0): u\in \bar{\Dc}((-\infty,
\infty))\}.
\]
 Moreover, the weak uniform tracking property holds.
\end{theorem}

\begin{theorem}\label{t:As0NSE}
If  $g_0$ is normal in $L^2_{\mathrm{loc}}(\mathbb{R}; V')$ and
every complete trajectory of $\bar{\Dc}$ is strongly
continuous, then the weak global attractor $\Aw$ is a strongly
compact strong global attractor $\As$. Moreover, the strong
uniform tracking property holds.
\end{theorem}

Finally, we obtain the trajectory attractor for 3D NSE  with a fixed time-dependent force $g_0$ due to Theorems \ref{t:weaktrattracor} and \ref{t:strongtrattracor}.

\begin{theorem}\label{t:traattractor}
The trajectory attractor for 3D NSE with force $g_0$ exists and
\[\mathfrak A=\Pi_+\bar{\Dc}((-\infty, \infty))=\{u(\cdot)|_{[0,\infty)}:u\in \bar{\Dc}((-\infty, \infty))\}, \]
satisfying
\[\Aw=\mathfrak A(t)=\{u(t): u\in \mathfrak A\},\quad \forall\, t\geq 0.\]
Furthermore, if $g_0$ is normal in $L^2_{\mathrm{loc}}(\mathbb{R}, V')$ and every complete trajectory of $\bar{\Dc}$ is strongly continuous then the trajectory attractor $\mathfrak A$ uniformly attracts $\Dc([0,\infty))$
in $L^\infty_{\mathrm{loc}}((0,\infty);H)$.
\end{theorem}

\section{Open problems}\label{s:OP}
In this section we assume that  $g_0$ is translation compact in
$L^2_\mathrm{loc}(\mathbb{R};V')$ and denote by
\[
\bar{\Sigma}:=\overline{\{g_0(\cdot+h)|h\in \mathbb R\}}^{L^2_\mathrm{loc}(\mathbb{R};V')}.
\]
 Note that the class of translation compact functions is also a  closed subspace of
the class of translation bounded functions, but it is a
proper subset of the class of normal functions (for more details, see \cite{LWZ05}). Note that the argument in Section \ref{3DNSE} before Lemma \ref{l:precompactofLH} is valid for $\Sigma$ replaced by $\bar{\Sigma}$ and Lemma \ref{l:precompactofLH} can be improved as follows.

\begin{lemma} \label{l:compactofLH}
Let $u_n(t)$ be a sequence of Leray-Hopf solutions of  \eqref{NSE1} with forces
$g_n \in \bar{\Sigma}$, such that $u_n(t) \in X$ for all $t\geq t_1$. Then
\begin{equation*}
\begin{split}
u_n \ \ &\mbox{is bounded in} \ \ L^2(t_1,t_2;V)\ \ \mbox{and} \ \ L^\infty(t_1,t_2;H),\\
\ddt u_n \ \  &\mbox{is bounded in} \ \ L^{4/3}(t_1,t_2;V'),
\end{split}
\end{equation*}
for all $t_2>t_1$. Moreover, there exists a subsequence $n_j$, such
that $g_{n_j}$ converges in
$L^{2}_{\mathrm{loc}}(\mathbb{R};V')$  to some  $g \in
\bar{\Sigma}$ and $u_{n_j}$ converges in $C([t_1, t_2]; \Hw)$ to some
Leray-Hopf solution $u(t)$ of  \eqref{NSE1} with the force $g$,
i.e.,
\[
(u_{n_j},v) \to (u,v) \quad
\mbox{uniformly on} \quad  [t_1,t_2],
\]
as $n_j\to \infty$, for all $v \in H$.
\end{lemma}

\begin{proof}See \cite{CF89, CV02}. Here we give a brief sketch.

The proof of Lemma \ref{l:precompactofLH} is still valid if we substitute $\bar{\Sigma}$ for $\Sigma$. So, the remains is  to verify \eqref{EI} for the limit $u$. We have
\begin{equation} \label{EIofun}
|u_n(t)|^2 + 2\nu \int_{t_0}^t \|u_n(s)\|^2 \, ds \leq
|u_n(t_0)|^2 + 2\int_{t_0}^t \langle g_n(s), u_n(s)\rangle \, ds
\end{equation}
for all $t \geq t_0$, $t_0$ a.e. in $[t_1,\infty)$. Note that
\[
\begin{aligned}
u_{n}(t)\rightarrow u(t) \quad &\text{ weakly in } H, \quad \forall\, t\geq t_1,\\
&\text{ strongly in } H, \quad t \text{ a.e. in } [t_1,\infty),\\
&\text{ weakly in }L^2_{\mathrm{loc}}(t_1,\infty;V),
\end{aligned}
\] and the convergence in \eqref{convofgn} is strong for $g_0$ is translation compact in
$L^2_\mathrm{loc}(\mathbb{R};V')$. Therefore, taking the limit of \eqref{EIofun} we obtain the energy inequality
\begin{equation*}
|u(t)|^2 + 2\nu \int_{t_0}^t \|u(s)\|^2 \, ds \leq
|u(t_0)|^2 + 2\int_{t_0}^t \langle g(s), u(s)\rangle \, ds
\end{equation*}
for all $t \geq t_0$, $t_0$ a.e. in $[t_1,\infty)$.
\end{proof}
Due to this lemma, now we can consider another evolutionary system with $\bar{\Sigma}$ as a symbol space.  The family
of trajectories of the evolutionary system  consists of all Leray-Hopf solutions of the family
of 3D Navier-Stokes equations with  forces $g\in \bar{\Sigma}$ in $X$:
\[
\begin{split}
\Dc_{\bar{\Sigma}}([T,\infty)) := \{&u(\cdot): u(\cdot) \mbox{ is a
Leray-Hopf}
\mbox{ solution on } [T,\infty)\\
&\mbox{with the force } g\in \bar{\Sigma} \mbox{ and } u(t) \in X, \
\forall t \in [T,\infty)\}, \quad T \in \mathbb{R},
\end{split}
\]
\[
\begin{split}
\Dc_{\bar{\Sigma}}((-\infty,\infty)) := \{&u(\cdot): u(\cdot)
\mbox{ is a Leray-Hopf} \ \mbox{ solution on } (-\infty,\infty)\\
&\mbox{with the force } g\in \bar{\Sigma} \mbox{ and } u(t) \in X, \
\forall t \in (-\infty,\infty)\}.
\end{split}
\]
Obviously, $\Dc\subset \Dc_{\bar{\Sigma}}$.

We have the following lemma.
\begin{lemma} \label{l:compact}
The evolutionary system $\Dc_{\bar{\Sigma}}$ of the family of  3D NSE with
forces in $\bar{\Sigma}$ satisfies \={A1}, \={A2} and \={A3}.
\end{lemma}
\begin{proof}
The proof is similar to Lemma \ref{l:compact0}. The difference is that we have
to use Lemma \ref{l:compactofLH} instead of Lemma \ref{l:precompactofLH}, and that $\{u^i\}$ and $u$ would now be contained in $\Dc_{\bar{\Sigma}}([0,\infty))$.
\end{proof}
Similarly,  Lemma \ref{l:compact} and Theorem \ref{t:weakA}
 yield the following (cf. \cite{CV02}).
\begin{theorem} \label{t:AwNSE}The uniform weak global attractor $\Aw^{\bar{\Sigma}}$ for the family of 3D
NSE with forces $g\in \bar{\Sigma}$ exists, $\Aw^{\bar{\Sigma}}$ is the maximal invariant
and maximal quasi-invariant set w.r.t.  the corresponding
evolutionary system $\Dc_{\bar{\Sigma}}$, and
\[
\Aw^{\bar{\Sigma}} =\{u(0): u \in \Dc_{\bar{\Sigma}}((-\infty, \infty))\}.
\]
Moreover, the weak uniform tracking property holds.
\end{theorem}

Theorem \ref{t:strongA} and \ref{t:AComp} give a criterion for strong compactness of the attractor.
\begin{theorem}\label{t:AsNSE}
If every complete trajectory of the family of 3D NSE with forces $g\in
\bar{\Sigma}$ is strongly continuous, then the weak global attractor $\Aw^{\bar{\Sigma}}$
is a strongly compact strong global attractor $\As^{\bar{\Sigma}}$. Moreover, the
strong uniform tracking property holds.
\end{theorem}
Let $\bar{\Dc}$ be the closure of the  evolutionary  system $\Dc$.
Obviously, $\Dc\subset\bar{\Dc}\subset\Dc_{\bar{\Sigma}}$. Then, an
interesting problem arises:
\begin{open}\label{op:attractor}
Are the uniform global attractors $\Ab$ and
$\Ab^{\bar{\Sigma}}$ in Theorems \ref{t:Aw0NSE}
and \ref{t:AsNSE} identical?
\end{open}
If the solutions of 3D NSE are unique, then the answer is positive
due to Theorem \ref{t:A0AwA-weak} and \ref{t:A0AwA-strong1}.
However, the negative answer, i.e., $\Ab\subsetneq\Ab^{\bar{\Sigma}}$, would
imply that the Leray-Hopf weak solutions are not unique and the uniform (w.r.t. symbol space) attractor doesn't satisfy the minimality property with respect to uniformly (w.r.t. initial time) attracting for  the original 3D NSE with fixed external force $g_0$.

We can also obtain a trajectory attractor for $\Dc_{\bar{\Sigma}}$ as in Section 5:

\begin{theorem}\label{t:traattractorsigma}
The trajectory attractor for  the family of 3D
NSE with forces $g\in \bar{\Sigma}$  exists and
\[\mathfrak A^{\bar{\Sigma}}=\Pi_+\Dc_{\bar{\Sigma}}((-\infty, \infty))=\{u(\cdot)|_{[0,\infty)}:u\in \Dc_{\bar{\Sigma}}((-\infty, \infty))\}, \]
satisfying
\[\mathcal A^{\bar{\Sigma}}_{\mathrm{w}}=\mathfrak A^{\bar{\Sigma}}(t)=\{u(t): u\in \mathfrak A^{\bar{\Sigma}}\},\quad \forall\, t\geq 0.\]
Furthermore, if  every complete trajectory of $\Dc_{\bar{\Sigma}}$ is strongly continuous then $\mathfrak A^{\bar{\Sigma}}$ uniformly attracts $\Dc_{\bar{\Sigma}}([0,\infty))$
in $L^\infty_{\mathrm{loc}}((0,\infty);H)$.
\end{theorem}
A similar problem on the relationship of this trajectory attractor and that for $\Dc$ also arises:
\begin{open}\label{op:traattractor}
Are the trajectory attractors $\mathfrak A$ and
$\mathfrak A^{\bar{\Sigma}}$ in Theorem \ref{t:traattractor} and \ref{t:traattractorsigma} identical?
\end{open}

 This open problem hints that, in general, the trajectory attractors constructed in  \cite{CV02} for the systems without uniqueness might not satisfy the minimality
property.

\section*{Acknowledgements}
The authors are very grateful to referees for several valuable comments.

\end{document}